\documentclass[11pt]{amsart}

\usepackage[T1]{fontenc}
\usepackage[utf8]{inputenc}
\usepackage{lmodern}
\usepackage{microtype}
\usepackage{amsmath,amssymb,amsthm,mathtools}
\usepackage{booktabs}
\usepackage[hidelinks]{hyperref}
\hypersetup{
  pdftitle={Ptolemaic negative type and the values q(6) and q(7)},
  pdfauthor={Eren Ercan},
  pdfsubject={Ptolemaic negative type and triangular-hyperfield thresholds},
  pdfkeywords={Lorentzian polynomial, triangular hyperfield, uniform matroid, Ptolemaic metric, negative type, Schoenberg matrix, metric inversion}
}
\usepackage{aliascnt}
\usepackage[nameinlink,capitalise]{cleveref}

\newtheorem{theorem}{Theorem}[section]
\newaliascnt{proposition}{theorem}
\newtheorem{proposition}[proposition]{Proposition}
\aliascntresetthe{proposition}
\newaliascnt{lemma}{theorem}
\newtheorem{lemma}[lemma]{Lemma}
\aliascntresetthe{lemma}
\newaliascnt{corollary}{theorem}
\newtheorem{corollary}[corollary]{Corollary}
\aliascntresetthe{corollary}
\theoremstyle{definition}
\newaliascnt{definition}{theorem}
\newtheorem{definition}[definition]{Definition}
\aliascntresetthe{definition}
\theoremstyle{remark}
\newaliascnt{remark}{theorem}
\newtheorem{remark}[remark]{Remark}
\aliascntresetthe{remark}

\crefname{theorem}{theorem}{theorems}
\Crefname{theorem}{Theorem}{Theorems}
\crefname{proposition}{proposition}{propositions}
\Crefname{proposition}{Proposition}{Propositions}
\crefname{lemma}{lemma}{lemmas}
\Crefname{lemma}{Lemma}{Lemmas}
\crefname{corollary}{corollary}{corollaries}
\Crefname{corollary}{Corollary}{Corollaries}
\crefname{definition}{definition}{definitions}
\Crefname{definition}{Definition}{Definitions}
\crefname{remark}{remark}{remarks}
\Crefname{remark}{Remark}{Remarks}

\newcommand{\R}{\mathbb R}
\newcommand{\T}{\mathbb T}
\newcommand{\Gr}{\operatorname{Gr}}
\newcommand{\PL}{\mathbb P\operatorname{L}}
\newcommand{\Qform}{\mathcal Q}
\newcommand{\CS}{\operatorname{CS}}
\newcommand{\qstar}{q_*}
\newcommand{\Her}{\mathcal H}

\title[Ptolemaic negative type and $q(6),q(7)$]
{Ptolemaic negative type and the values $q(6)$ and $q(7)$}

\author{Eren Ercan}
\address{Independent researcher, Sydney, Australia}
\email{eren321@gmail.com}
\date{}


\begin{document}

\begin{abstract}
Baker, Huh, Kummer, and Lorscheid define the triangular-hyperfield
threshold $q(n)=q(U_{2,n})$ and conjecture exact values for all $n$.
Using their identity $q(n)=P(n-1)$, where $P(m)$ is the universal
negative-type exponent of $m$-point Ptolemaic metrics, we prove
\[
 q(6)=\log_2\frac94,
 \qquad
 q(7)=1.
\]
For five points, we separate zero-sum coefficient vectors by sign
pattern. A sharp two-summand inequality proves the inequality for the
$1+4$ pattern. For the $2+3$ pattern, a sharp four-point partial-correlation
bound and an exact copositivity identity prove the required inequality.
For six points, we construct an involutive coefficient transport under
metric inversion that preserves the negative-type quadratic form. Applying
the transport at an index with positive local contribution converts a
hypothetical $3+3$ counterexample into a $2+4$ counterexample. Complete split graph metrics attain both bounds.
\end{abstract}

\maketitle

\section{Introduction}

For a matroid $M$, Baker, Huh, Kummer, and Lorscheid define
\[
 q(M)=\sup\{q>0: \Gr_M^{\mathrm w}(\T_q)\subseteq\PL_M\},
\]
and write $q(n)=q(U_{2,n})$ for the rank-two uniform matroid.  They prove
$q(4)=2$ and $q(5)=\log_2 3$, and Conjecture~6.4 of \cite{BHKL}
predicts
\begin{equation}\label{eq:BHKL-conjecture}
 q(n)=
 \begin{cases}
  2\log_2\!\left(\dfrac{n}{n-2}\right),&n\text{ even},\\[2mm]
  \log_2\!\left(\dfrac{n+1}{n-3}\right),&n\text{ odd}.
 \end{cases}
\end{equation}
The first values left open by their paper are therefore
\[
 q(6)=\log_2\frac94,
 \qquad
 q(7)=1.
\]

A metric space $(X,d)$ is \emph{Ptolemaic} if
\begin{equation}\label{eq:ptolemy-definition}
 d(x,y)d(z,w)
 \le d(x,z)d(y,w)+d(x,w)d(y,z)
\end{equation}
for all $x,y,z,w\in X$. For $m\ge3$, let
\begin{equation}\label{eq:P-definition}
 \begin{split}
 P(m)=\sup\{q>0:{}&\text{ every Ptolemaic metric on exactly $m$ points}\\
                  &\text{ has $q$-negative type}\}.
 \end{split}
\end{equation}
Within an $m$-point metric, zero coefficients restrict the quadratic form
to any chosen subset. Proposition~7.5 of \cite{BHKL}
gives the exact correspondence
\begin{equation}\label{eq:correspondence}
 q(n)=P(n-1).
\end{equation}
Our main result is the following.

\begin{theorem}\label{thm:main}
One has
\[
 P(5)=\log_2\frac94,
 \qquad
 P(6)=1.
\]
Equivalently,
\[
 q(6)=\log_2\frac94,
 \qquad
 q(7)=1.
\]
Thus Conjecture~6.4 of \cite{BHKL} holds for $n=6$ and $n=7$.
\end{theorem}

For the five-point statement, a sharp four-point theorem bounds the
normalized off-diagonal entries of a Schur complement. At
$q=\log_2(9/4)$, its lower partial-correlation bound is $-1/2$, the
copositivity threshold for three nonnegative coordinates. For six points,
metric inversion transports a zero-sum coefficient vector without changing
its quadratic form. Choosing an inversion point whose local contribution is
positive changes a balanced sign pattern into an adjacent one.

The only computer-assisted component is an exact rational Bernstein
certificate in the proof of the four-point correlation theorem. The
ancillary files contain the complete coefficient list, the generator, and
an independent reconstruction verifier. All remaining arguments are
analytic.

\subsection*{Statement on AI usage}
AI tools were used during proof exploration, symbolic verification, and
preparation of the manuscript and exact certificate. The author reviewed
and checked the arguments and takes full responsibility for all content.

\section{Negative type and a sharp two-summand inequality}

Following the standard Schoenberg formulation \cite{Schoenberg}, a finite metric space $(X,d)$ has \emph{$q$-negative type} if
\[
 \sum_{x,y\in X}z_xz_y d(x,y)^q\le0
 \qquad\text{whenever}\qquad
 \sum_{x\in X}z_x=0.
\]
It has \emph{strict $q$-negative type} if the inequality is strict for
every nonzero zero-sum vector.
Equivalently, if $a_i,b_j>0$ have equal total mass, then
\begin{equation}\label{eq:generalized-roundness}
 \sum_{i<i'}a_ia_{i'}d(x_i,x_{i'})^q
 +\sum_{j<j'}b_jb_{j'}d(y_j,y_{j'})^q
 \le
 \sum_{i,j}a_ib_jd(x_i,y_j)^q.
\end{equation}
We normalize both total masses to one. After zero coefficients are
removed, we say that the coefficient vector has \emph{sign pattern}
$a+b$ if it has $a$ positive and $b$ negative entries. A sign pattern is
\emph{balanced} when $a=b$.

For $1\le q\le2$, put
\[
 h_q=2^q-2.
\]

\begin{lemma}[Sharp two-summand estimate]\label{lem:two-summand}
For $s,t\ge0$ and $1\le q\le2$,
\begin{equation}\label{eq:two-summand}
 (s+t)^q\le s^q+t^q+h_q(st)^{q/2}.
\end{equation}
The coefficient $h_q$ is sharp. If $1<q<2$ and $s,t>0$, equality holds precisely when $s=t$.
\end{lemma}

\begin{proof}
For $t>0$, set $z=s/t$ and define
\[
 \eta_q(z)=\frac{(1+z)^q-1-z^q}{z^{q/2}}.
\]
The symmetry $\eta_q(z)=\eta_q(z^{-1})$ reduces the problem to $z\ge1$. Write $z=e^{2x}$ with $x\ge0$. Then
\[
 \eta_q(e^{2x})=2^q\cosh^q x-2\cosh(qx).
\]
For $x>0$, the inequality $\frac{d}{dx}\eta_q(e^{2x})\le0$ is equivalent to
\begin{equation}\label{eq:hyperbolic-comparison}
 \frac{\sinh(qx)}{\sinh x}\ge(2\cosh x)^{q-1}.
\end{equation}
For fixed $x>0$, let
\[
 F_x(q)=\log\frac{\sinh(qx)}{\sinh x}-(q-1)\log(2\cosh x).
\]
One has $F_x(1)=F_x(2)=0$ and
\[
 F_x''(q)=-x^2\operatorname{csch}^2(qx)<0.
\]
Hence $F_x(q)\ge0$ for $1\le q\le2$, which proves \eqref{eq:hyperbolic-comparison}. Therefore $\eta_q$ is maximized at $z=1$, where $\eta_q(1)=2^q-2$. This proves \eqref{eq:two-summand}. The strict equality statement follows from strict concavity of $F_x$ for $x>0$. The cases $s=0$, $t=0$, $q=1$, and $q=2$ follow directly.
\end{proof}

\section{Sign patterns with one positive coefficient}

\begin{theorem}[One-anchor inequality]\label{thm:one-anchor}
Let $1\le q\le2$, and let $(X,d)$ be a metric space. Suppose a zero-sum coefficient vector has one positive coefficient and $r\ge2$ negative coefficients. If
\begin{equation}\label{eq:one-condition}
 h_q\le\frac{2}{r-1},
\end{equation}
then its $q$-negative-type inequality holds. Equivalently, every such sign pattern is valid for
\[
 q\le\log_2\frac{2r}{r-1}.
\]
The exponent is sharp on the equal-edge $r$-leaf star.
\end{theorem}

\begin{proof}
Normalize the positive coefficient to one and write the negative coefficients as $c_1,\ldots,c_r\ge0$, where $\sum_i c_i=1$. Put
\[
 \rho_i=d(o,y_i),
 \qquad
 u_i=c_i\rho_i^{q/2}.
\]
The triangle inequality and \cref{lem:two-summand} give
\[
 d(y_i,y_j)^q
 \le \rho_i^q+\rho_j^q+h_q(\rho_i\rho_j)^{q/2}.
\]
Consequently, the gap in \eqref{eq:generalized-roundness} satisfies
\begin{align}
 \Delta
 &=\sum_i c_i\rho_i^q-\sum_{i<j}c_ic_jd(y_i,y_j)^q\notag\\
 &\ge
 \left(1+\frac{h_q}{2}\right)\sum_i u_i^2
 -\frac{h_q}{2}\left(\sum_i u_i\right)^2\notag\\
 &=
 \left(1-\frac{(r-1)h_q}{2}\right)\sum_i u_i^2
 +\frac{h_q}{2}\sum_{i<j}(u_i-u_j)^2.
 \label{eq:one-stability}
\end{align}
The last identity uses
\[
 \sum_{i<j}(u_i-u_j)^2
 =r\sum_i u_i^2-\left(\sum_i u_i\right)^2.
\]
This proves the result under \eqref{eq:one-condition} and records the quantitative distance from equality.

For an equal-edge star with equal negative weights, the triangle estimate,
the sharp power estimate, and both terms in \eqref{eq:one-stability} are
sharp at the stated exponent.
\end{proof}

\begin{proposition}[One-positive sign patterns below exponent one]
\label{prop:one-anchor-subunit}
Let $0<q\le1$. Every zero-sum coefficient vector with one positive
coefficient satisfies the $q$-negative-type inequality in every metric
space.
\end{proposition}

\begin{proof}
Normalize the positive coefficient to one, write the negative
coefficients as $c_1,\ldots,c_r\ge0$ with $\sum_i c_i=1$, and put
$\rho_i=d(o,y_i)$. Concavity and the triangle inequality give
\[
 d(y_i,y_j)^q\le(\rho_i+\rho_j)^q\le\rho_i^q+\rho_j^q.
\]
Therefore
\[
 \sum_{i<j}c_ic_jd(y_i,y_j)^q
 \le\sum_i c_i(1-c_i)\rho_i^q
 \le\sum_i c_i\rho_i^q,
\]
which is the required inequality.
\end{proof}

\section{Partial correlations in four-point Ptolemaic metrics}

\begin{definition}[Partial correlation]\label{def:section-correlation}
Let $P,A,B,O$ be distinct points, and let $G$ be the Schoenberg matrix of
$d^{q/2}$ based at $P$:
\[
 G_{XY}=\frac{d(P,X)^q+d(P,Y)^q-d(X,Y)^q}{2}.
\]
For $X,Y\in\{A,B\}$, set
\[
 S_{XY}=G_{XY}-\frac{G_{XO}G_{OY}}{G_{OO}}.
\]
When $S_{AA}S_{BB}>0$, define the \emph{partial correlation}
\begin{equation}\label{eq:section-correlation-definition}
 \chi_q(A,B\mid P;O)
 =\frac{S_{AB}}{\sqrt{S_{AA}S_{BB}}}.
\end{equation}
When $G$ is positive semidefinite, this is the normalized inner product
of the projections of the vectors for $A$ and $B$ onto the orthogonal
complement of the vector for $O$. Degenerate cases used below are
understood by continuity.
\end{definition}

Let $A,B,U,V$ be four points of a Ptolemaic metric. Put
\[
 D=d(A,B)^q,
 \qquad
 R_U=d(A,U)^q,
 \qquad
 T_U=d(B,U)^q,
\]
and define $R_V,T_V$ similarly. The Schoenberg Gram matrix based at $A$, with $B$ placed first, has the block form
\[
 G=
 \begin{pmatrix}
  D&g^{\mathsf T}\\
  g&H
 \end{pmatrix},
 \qquad
 g_U=\frac{D+R_U-T_U}{2},
\]
where
\[
 H_{UU}=R_U,
 \qquad
 H_{UV}=\frac{R_U+R_V-d(U,V)^q}{2}.
\]
Define the Schur complement
\begin{equation}\label{eq:transverse}
 S=H-D^{-1}gg^{\mathsf T}.
\end{equation}
This is the Schur complement for the points $U,V$ after conditioning on
$B$ in the Schoenberg matrix based at $A$. For $q\le2$, each $S_{UU}$ is
nonnegative. Indeed,
$d^{q/2}$ is a metric and every three-point metric is Euclidean.

For $1\le q\le\log_2 3$, set
\begin{equation}\label{eq:kappa-lambda}
 \kappa_q=\frac{2^{q+1}-3}{3},
 \qquad
 \lambda_q=\frac{(2^q-2)^2}{4-(2^q-2)^2}.
\end{equation}
In this range, $0\le\lambda_q\le\kappa_q$.

The following sharp four-point estimate is the main ingredient. Its proof
is given in Appendix~\ref{app:geodesic}. The computer-assisted component is
an exact rational Bernstein certificate supplied with the ancillary files.

\begin{theorem}[Partial correlation for a geodesic insertion]\label{thm:geodesic-section}
Let $A,P,B,O$ be four distinct points in a Ptolemaic metric and
suppose that
\[
 d(A,B)=d(A,P)+d(P,B).
\]
With the notation of \cref{def:section-correlation}, let
$\chi_q=\chi_q(A,B\mid P;O)$. Then
\begin{equation}\label{eq:geodesic-section-bound}
 -\kappa_q\le\chi_q\le\lambda_q,
 \qquad
 1\le q\le\log_2 3.
\end{equation}
The lower constant is sharp. At $q=\log_2(9/4)$, it equals $-1/2$.
\end{theorem}

Appendix~\ref{app:geodesic} gives the endpoint reduction. The exact
certificate is a tensor-product Bernstein expansion with $24{,}305$
nonnegative rational coefficients. The ancillary files contain the
complete coefficient list, generator, and independent verifier.

\begin{proposition}[Four-point Schur-complement bound]\label{prop:transverse}
Let $1\le q\le\log_2 3$. For every four-point Ptolemaic metric and
every distinguished pair $A,B$, the matrix in \eqref{eq:transverse}
satisfies
\begin{equation}\label{eq:pair-lower}
 S_{UV}\ge-\kappa_q\sqrt{S_{UU}S_{VV}}.
\end{equation}
The constant is sharp.
\end{proposition}

\begin{proof}
Write
\[
 \begin{aligned}
 r_U&=d(A,U),& s_U&=d(B,U),\\
 r_V&=d(A,V),& s_V&=d(B,V),& d&=d(A,B).
 \end{aligned}
\]
and let $e=d(U,V)$. The exact upper endpoint of the simultaneous triangle and Ptolemy interval for $e$ is
\begin{equation}\label{eq:upper-endpoint}
 e_+=\min\left\{
 r_U+r_V,
 s_U+s_V,
 \frac{r_Us_V+r_Vs_U}{d}
 \right\}.
\end{equation}
The given value satisfies $e\le e_+$. Replacing $e$ by $e_+$ preserves all triangle and Ptolemy inequalities, because every lower bound on $e$ was already satisfied by the original value. The diagonal entries of $S$ remain fixed, while
\[
 S_{UV}=\frac{R_U+R_V-e^q}{2}-\frac{g_Ug_V}{D}
\]
decreases. It therefore suffices to prove \eqref{eq:pair-lower} at $e=e_+$.

If $e_+=r_U+r_V$, then $A$ lies on a geodesic from $U$ to $V$, and
\cref{thm:geodesic-section} applies directly, with $B$ as the conditioning
point. If $e_+=s_U+s_V$, the same argument uses $B$ as the geodesic middle
point. The Schur complement is unchanged when the two conditioning points
are exchanged: direct substitution gives
\[
 S^{A\mid B}=S^{B\mid A}.
\]

In the remaining case,
\[
 d\,e_+=r_Us_V+r_Vs_U.
\]
Metric inversion preserves Ptolemaic metrics. Invert the metric at $A$. The Ptolemy equality becomes
\[
 \widehat d(U,V)=\widehat d(U,B)+\widehat d(B,V),
\]
so $B$ is a geodesic middle point. Under this inversion,
\[
 \widehat S_{UU}=\frac{S_{UU}}{R_U^2},
 \qquad
 \widehat S_{VV}=\frac{S_{VV}}{R_V^2},
 \qquad
 \widehat S_{UV}=\frac{S_{UV}}{R_UR_V}.
\]
Hence the normalized off-diagonal entry is invariant, and
\cref{thm:geodesic-section} applies again.

Cases in which a diagonal entry of $S$ vanishes follow by continuity. The
balanced crossing configuration described in Appendix~\ref{app:geodesic}
attains the lower bound.
\end{proof}

\section{Two-point sign support and copositivity}

\begin{lemma}[Copositive Schur closure]\label{lem:copositive-closure}
Let $S=(S_{ij})_{i,j=1}^r$ be symmetric with $S_{ii}\ge0$, and suppose that
\[
 S_{ij}\ge-\theta\sqrt{S_{ii}S_{jj}}
 \qquad(i\ne j)
\]
for some $\theta\ge0$. For $c\in\R_+^r$, put
\[
 x_i=c_i\sqrt{S_{ii}},
 \qquad
 \varepsilon_{ij}=S_{ij}+\theta\sqrt{S_{ii}S_{jj}}\ge0.
\]
Then
\begin{align}
 c^{\mathsf T}Sc
 ={}&
 \bigl(1-(r-1)\theta\bigr)\sum_{i=1}^r x_i^2
 +\theta\sum_{i<j}(x_i-x_j)^2\notag\\
 &+2\sum_{i<j}c_ic_j\varepsilon_{ij}.
 \label{eq:copositive-stability}
\end{align}
In particular, $S$ is copositive whenever $\theta\le1/(r-1)$.
\end{lemma}

\begin{proof}
Expand $c^{\mathsf T}Sc$, substitute the definition of $\varepsilon_{ij}$, and use
\[
 \sum_{i<j}(x_i-x_j)^2
 =(r-1)\sum_i x_i^2-2\sum_{i<j}x_ix_j.
\]
Every term on the right of \eqref{eq:copositive-stability} is nonnegative under the stated bound on $\theta$.
\end{proof}

\begin{theorem}[Two-anchor inequality]\label{thm:two-anchor}
Let $(X,d)$ be a Ptolemaic metric. Suppose a zero-sum coefficient vector has two positive coefficients and $r\ge2$ negative coefficients. Let $1\le q\le\log_2 3$. If
\begin{equation}\label{eq:kappa-condition}
 \kappa_q\le\frac1{r-1},
\end{equation}
then its $q$-negative-type inequality holds.
\end{theorem}

\begin{proof}
Write the positive coefficients as $a,b>0$ and the negative coefficients as $c_1,\ldots,c_r>0$, normalized by
\[
 a+b=\sum_{j=1}^r c_j=1.
\]
Use the notation of the preceding section with distinguished points $A,B$
and remaining points $U_1,\ldots,U_r$. Let $c=(c_1,\ldots,c_r)^{\mathsf T}$. A direct expansion of the gap in \eqref{eq:generalized-roundness} gives the exact identity
\begin{equation}\label{eq:gap-decomposition}
 \Delta
 =c^{\mathsf T}Sc
 +D^{-1}\bigl(bD-g^{\mathsf T}c\bigr)^2.
\end{equation}
Thus positive semidefiniteness of $S$ is unnecessary. The coefficient vector $c$ lies in the nonnegative orthant, so copositivity of $S$ suffices.

For $j<k$, set
\[
 \varepsilon_{jk}
 =S_{jk}+\kappa_q\sqrt{S_{jj}S_{kk}}\ge0.
\]
Applying \cref{lem:copositive-closure} with $\theta=\kappa_q$ gives the exact decomposition
\begin{align}
 \Delta={}&
 D^{-1}\bigl(bD-g^{\mathsf T}c\bigr)^2
 +\bigl(1-(r-1)\kappa_q\bigr)\sum_{j=1}^r x_j^2\notag\\
 &+\kappa_q\sum_{j<k}(x_j-x_k)^2
 +2\sum_{j<k}c_jc_k\varepsilon_{jk},
 \label{eq:two-anchor-stability}
\end{align}
where $x_j=c_j\sqrt{S_{jj}}$. Under \eqref{eq:kappa-condition}, every term on the right is nonnegative.
\end{proof}

At the six-element exponent, $r=3$ and $\kappa_q=1/2$. The copositive term becomes
\begin{equation}\label{eq:three-variable-copositive}
 \sum_{j=1}^3x_j^2-\sum_{j<k}x_jx_k
 =\frac12\sum_{j<k}(x_j-x_k)^2.
\end{equation}

\begin{definition}
For integers $p\ge1$ and $r\ge2$, let $\CS(p,r)$ denote the graph metric
of the complete split graph $K_p\vee\overline{K_r}$. Thus distances
within the clique part and between the two parts are $1$, while distances
between distinct vertices of the independent part are $2$.
\end{definition}

Every $\CS(p,r)$ is Ptolemaic. In a four-point restriction, each product
of opposite distances belongs to $\{1,2,4\}$. A product equal to $4$
requires all four points to lie in the independent part, in which case
all three products equal $4$. Otherwise the largest product is at most
$2$, while each of the other two is at least $1$.

\begin{corollary}[Sharp two-anchor complete-split constant]\label{cor:split-two}
Suppose
\[
 q_{2,r}=\log_2\frac{3r}{2(r-1)}
\]
belongs to $[1,\log_2 3]$. Then every Ptolemaic metric satisfies the negative-type inequality for every $2+r$ sign pattern whenever $1\le q\le q_{2,r}$. The exponent is sharp on $\CS(2,r)$.
\end{corollary}

\begin{proof}
At $q=q_{2,r}$,
\[
 \kappa_q=\frac{2^{q+1}-3}{3}=\frac1{r-1}.
\]
For sharpness, assign weights $1/2$ to the two vertices in the clique part and $1/r$ to the vertices in the independent part. Then \eqref{eq:generalized-roundness} becomes
\[
 \frac14+\frac{r-1}{2r}\,2^q\le1,
\]
which is equivalent to $2^q\le3r/[2(r-1)]$.
\end{proof}

Within the range $1\le q\le\log_2 3$, the first three sharp two-anchor values are
\[
\begin{array}{c|c|c}
 r&2^{q_{2,r}}&q_{2,r}\\ \hline
 2&3&\log_2 3\\
 3&9/4&\log_2(9/4)\\
 4&2&1.
\end{array}
\]

\section{The five-point theorem}

By definition, $P(5)$ concerns metrics on exactly five points. Set
\[
 \qstar=\log_2\frac94.
\]
Then
\[
 h_{\qstar}=\frac14,
 \qquad
 \kappa_{\qstar}=\frac12.
\]

\begin{theorem}[Five-point Ptolemaic negative type]\label{thm:five-point}
Every five-point Ptolemaic metric has $\qstar$-negative type, and the exponent is sharp. Hence
\[
 \boxed{P(5)=\log_2\frac94.}
\]
\end{theorem}

\begin{proof}
Let $z$ be a nonzero zero-sum coefficient vector. After replacing $z$ by $-z$ if necessary, the smaller sign support has size one or two.

For the $1+4$ sign pattern, \cref{thm:one-anchor} applies because
\[
 h_{\qstar}=\frac14<\frac23.
\]
For the $2+3$ sign pattern, \cref{thm:two-anchor} applies because
\[
 \kappa_{\qstar}=\frac12=\frac1{3-1}.
\]
After zero coefficients are removed, the same arguments apply to every proper subset; the $1+1$ case is immediate. Thus every zero-sum quadratic form is nonpositive.

The metric $\CS(2,3)$ is Ptolemaic. With weights $1/2,1/2$ on the clique part and $1/3,1/3,1/3$ on the independent part, \eqref{eq:generalized-roundness} becomes
\[
 \frac14+\frac13\,2^q\le1.
\]
This is equivalent to $2^q\le9/4$, so no larger universal exponent is possible.
\end{proof}

\begin{theorem}[Strictness below the endpoint and equality at the endpoint]\label{thm:five-point-stability}
Let $(X,d)$ be a Ptolemaic metric on five distinct points.

For every $1\le q<\qstar$, the metric has strict $q$-negative type. At $q=\qstar$, a nonzero null vector must have sign pattern $2+3$. With the notation of \cref{thm:two-anchor}, equality is equivalent to
\begin{equation}\label{eq:equality-system}
 \begin{gathered}
 bD=g^{\mathsf T}c,\\
 c_1\sqrt{S_{11}}=c_2\sqrt{S_{22}}=c_3\sqrt{S_{33}},\\
 S_{ij}=-\frac12\sqrt{S_{ii}S_{jj}}\qquad(i<j).
 \end{gathered}
\end{equation}
The balanced complete split metric $\CS(2,3)$ realizes this system.
\end{theorem}

\begin{proof}
For the $1+4$ pattern, \eqref{eq:one-stability} has a strictly positive coefficient on $\sum u_i^2$ throughout $1\le q\le\qstar$. For the $2+3$ pattern, $\kappa_q<1/2$ when $q<\qstar$. Since $q<2$ and the points are distinct, every three-point $d^{q/2}$-triangle is nondegenerate, so $S_{ii}>0$. The coefficient
\[
 1-2\kappa_q
\]
in \eqref{eq:two-anchor-stability} is then strictly positive.

At $q=\qstar$, formula \eqref{eq:two-anchor-stability} reduces to
\begin{align*}
 \Delta={}&D^{-1}\bigl(bD-g^{\mathsf T}c\bigr)^2
 +\frac12\sum_{i<j}(x_i-x_j)^2
 +2\sum_{i<j}c_ic_j\varepsilon_{ij}.
\end{align*}
All weights and all $S_{ii}$ are positive. Equality therefore holds exactly under \eqref{eq:equality-system}. Direct substitution for $\CS(2,3)$ gives
\[
 D=1,
 \qquad
 S_{ii}=\frac34,
 \qquad
 S_{ij}=-\frac38,
\]
and the balanced coefficient vector $(1/2,1/2,-1/3,-1/3,-1/3)$ satisfies the remaining two conditions.
\end{proof}

\begin{corollary}[Six-element triangular-hyperfield threshold]\label{cor:q6}
With the notation $q(n)=q(U_{2,n})$ of \cite{BHKL}, one has
\[
 \boxed{q(6)=\log_2\frac94.}
\]
\end{corollary}

\begin{proof}
By Proposition~7.5 of Baker, Huh, Kummer, and Lorscheid, $q(n)=P(n-1)$ \cite{BHKL}. Apply \cref{thm:five-point} with $n=6$.
\end{proof}

\section{Ptolemaic inversion}

Let $(X,d)$ be a finite metric space and fix $p\in X$.  Its metric
inversion at $p$ is the function $d^{(p)}$ defined by
\begin{equation}\label{eq:inversion}
 \begin{aligned}
 d^{(p)}(p,p)&=0,\\
 d^{(p)}(p,x)&=\frac1{d(p,x)} &&(x\ne p),\\
 d^{(p)}(x,y)&=\frac{d(x,y)}{d(x,p)d(y,p)} &&(x,y\ne p).
 \end{aligned}
\end{equation}

\begin{lemma}[Ptolemaic inversion]\label{lem:ptolemy-inversion}
If $d$ is Ptolemaic, then $d^{(p)}$ is a Ptolemaic metric.
\end{lemma}

\begin{proof}
For $x,y\ne p$, the triangle inequality
\[
 d^{(p)}(x,y)\le d^{(p)}(x,p)+d^{(p)}(p,y)
\]
is equivalent, after multiplication by $d(x,p)d(y,p)$, to the triangle
inequality for $d$.  A triangle inequality for three points distinct from
$p$ becomes, after clearing the common positive denominator, a Ptolemy
inequality for those three points together with $p$.

A Ptolemy inequality for four points distinct from $p$ is multiplied by one
common positive factor under \eqref{eq:inversion}.  A Ptolemy inequality
involving $p$ becomes a triangle inequality for $d$.  Hence all metric and
Ptolemy inequalities are preserved.
\end{proof}

For $q>0$ and $c\in\R^X$ with $\sum_xc_x=0$, define
\begin{equation}\label{eq:quadratic-form}
 \Qform_q(d;c)=\sum_{x,y\in X}c_xc_y d(x,y)^q
\end{equation}
and the local potential
\begin{equation}\label{eq:potential}
 \Phi_c(x)=\sum_{y\in X}c_y d(x,y)^q.
\end{equation}
When two metrics occur simultaneously, we write $\Phi_c^d$ to indicate
the metric used in this definition.  Then
\begin{equation}\label{eq:potential-sum}
 \Qform_q(d;c)=\sum_{x\in X}c_x\Phi_c(x).
\end{equation}

\section{Coefficient transport and support reduction}

\begin{theorem}[Quadratic-form covariance]\label{thm:transport}
Let $d$ be a metric on $X$, let $q>0$, and let $c\in\R^X$ satisfy
$\sum_xc_x=0$.  For $p\in X$, define $T_pc\in\R^X$ by
\begin{equation}\label{eq:transport}
 (T_pc)_x=c_xd(x,p)^q\quad(x\ne p),
 \qquad
 (T_pc)_p=-\Phi_c(p).
\end{equation}
Then
\begin{equation}\label{eq:transport-zero}
 \sum_x(T_pc)_x=0
\end{equation}
and
\begin{equation}\label{eq:transport-invariance}
 \boxed{\Qform_q(d^{(p)};T_pc)=\Qform_q(d;c).}
\end{equation}
If $c_p\Phi_c(p)>0$, then the sign of the coefficient at $p$ is reversed by
$T_p$, while the sign of every other nonzero coefficient is preserved.
Moreover, if $c'=T_pc$ and $d'=d^{(p)}$, then
\begin{align}
 \Phi_{c'}^{d'}(p)&=-c_p,\label{eq:potential-transform-p}\\
 \Phi_{c'}^{d'}(x)
 &=\frac{\Phi_c^d(x)-\Phi_c^d(p)-c_p d(x,p)^q}{d(x,p)^q}
 \qquad(x\ne p),\label{eq:potential-transform-x}
\end{align}
and transport is involutive:
\begin{equation}\label{eq:transport-involution}
 T_p^{d'}(T_p^d c)=c.
\end{equation}
\end{theorem}

\begin{proof}
Since $d(p,p)=0$,
\[
 \Phi_c(p)=\sum_{x\ne p}c_xd(x,p)^q,
\]
which proves \eqref{eq:transport-zero}.  The terms in
\eqref{eq:transport-invariance} indexed by $x,y\ne p$ satisfy
\[
 (T_pc)_x(T_pc)_y d^{(p)}(x,y)^q
 =c_xc_y d(x,y)^q.
\]
The terms containing $p$ contribute
\begin{align*}
 2(T_pc)_p\sum_{x\ne p}(T_pc)_x d^{(p)}(p,x)^q
 &=2(-\Phi_c(p))\sum_{x\ne p}c_x\\
 &=2c_p\Phi_c(p),
\end{align*}
which is the contribution of the $p$-row and $p$-column to
$\Qform_q(d;c)$.  This proves the identity.

For $x\ne p$, multiplication by $d(x,p)^q>0$ preserves the sign.  If
$c_p\Phi_c(p)>0$, then $-\Phi_c(p)$ has the opposite sign from $c_p$.
The displayed formulas for the transformed potential follow by direct
substitution.  In particular,
\[
 \Phi_{c'}^{d'}(p)
 =\sum_{x\ne p}c_xd(x,p)^q d'(p,x)^q
 =\sum_{x\ne p}c_x=-c_p.
\]
For $x\ne p$, the same substitution gives
\eqref{eq:potential-transform-x}.  Applying the transport definition once
more now returns every coefficient of $c$, proving
\eqref{eq:transport-involution}.
\end{proof}

The argument only uses closure of the metric class under inversion.

\begin{theorem}[Reduction to adjacent support sizes]\label{thm:neighbor-reduction}
Let $\mathcal C$ be a class of finite metrics closed under metric inversion,
fix $q>0$, and let $a,b\ge2$.  Suppose the $q$-negative-type inequality
holds throughout $\mathcal C$ for both sign patterns
\[
 (a-1)+(b+1)
 \qquad\text{and}\qquad
 (a+1)+(b-1).
\]
Then it also holds for the sign pattern $a+b$.
\end{theorem}

\begin{proof}
Assume that a zero-sum vector $c$ of sign pattern $a+b$ on a metric
$d\in\mathcal C$ satisfies $\Qform_q(d;c)>0$.  By
\eqref{eq:potential-sum}, some $p\in X$ satisfies
$c_p\Phi_c(p)>0$.  If $p$ belongs to the positive side, transport at $p$
produces an $(a-1)+(b+1)$ counterexample.  If $p$ belongs to the negative
side, it produces an $(a+1)+(b-1)$ counterexample.  In either case
\cref{thm:transport} preserves the positive value of the quadratic form,
and inversion-closure keeps the metric in $\mathcal C$, a contradiction.
\end{proof}

\begin{corollary}[Balanced-support reduction]\label{thm:balanced-reduction}
Let $\mathcal C$ be inversion-closed.  If the $q$-negative-type inequality
holds for the sign pattern $(k-1)+(k+1)$ and its reversal, then it holds for
the balanced sign pattern $k+k$.
\end{corollary}

\begin{proof}
Apply \cref{thm:neighbor-reduction} with $a=b=k$; the two adjacent
patterns are reversals of one another.
\end{proof}

For a fixed cardinality $m$, consider a metric class closed under inversion
and under snowflaking $d\mapsto d^\alpha$ for $0<\alpha\le1$, and let
$\tau_{m,a}$ denote the supremal exponent for which the $a+(m-a)$ sign
pattern is valid throughout the class.  The same theorem, together with
snowflake monotonicity, gives
\begin{equation}\label{eq:support-minimum-principle}
 \tau_{m,a}\ge
 \min\{\tau_{m,a-1},\tau_{m,a+1}\}
 \qquad(1<a<m-1).
\end{equation}

\begin{corollary}[Reduction of balanced sign patterns]\label{cor:balanced-redundant}
For Ptolemaic metrics on $2k$ points, the balanced $k+k$ negative-type
inequalities follow from the adjacent $(k-1)+(k+1)$ inequalities at the
same exponent.
\end{corollary}

\begin{corollary}[Parity reduction at fixed support]\label{cor:parity-support}
Fix a total support size and an exponent in an inversion-closed metric class.
If every sign pattern whose smaller side has odd cardinality is valid, then
every sign pattern is valid except possibly an even near-balanced pattern
$k+(k+1)$.  More precisely, an even pattern $a+b$ with $a\le b$ follows
from odd patterns whenever $b-a\ge2$ or $a=b$.
\end{corollary}

\begin{proof}
If $a$ is even and $b-a\ge2$, both adjacent patterns
$(a-1)+(b+1)$ and $(a+1)+(b-1)$ have odd smaller side, so
\cref{thm:neighbor-reduction} applies.  If $a=b$, the balanced reduction
uses the odd adjacent pattern $(a-1)+(a+1)$.  When $b=a+1$, one adjacent pattern is the original pattern with its signs
reversed, so this argument gives no further reduction. This is the only case
not covered by the preceding reductions.
\end{proof}

\section{The six-point theorem}

We now apply the inversion reduction at exponent one.  The one-anchor and
two-anchor inequalities proved above give the only adjacent support patterns
that are needed.

\begin{theorem}\label{thm:P6}
Every six-point Ptolemaic metric has $1$-negative type, and the exponent is
sharp.  Hence
\[
 \boxed{P(6)=1.}
\]
\end{theorem}

\begin{proof}
Let $c$ be a nonzero zero-sum coefficient vector.  Discard zero
coefficients and reverse all signs if necessary. The $1+1$ case is
immediate, and every remaining smaller sign support has size one, two,
or three.

For a $1+r$ pattern with $2\le r\le5$, \cref{thm:one-anchor} applies at
$q=1$.  For a $2+r$ pattern with $2\le r\le4$,
\cref{thm:two-anchor} applies; in the limiting case $r=4$ one has
\[
 \kappa_1=\frac13=\frac1{4-1}.
\]
The only remaining pattern is $3+3$.  By
\cref{thm:balanced-reduction}, a $3+3$ counterexample would produce a
$2+4$ counterexample after metric inversion, which has already been
excluded.  Thus every zero-sum quadratic form is nonpositive.

For sharpness, take the complete split metric $\CS(3,3)$. Distances
within the three-vertex clique and between the two parts are $1$, while
distances between distinct vertices of the independent part are $2$.
Give each clique vertex coefficient $1/3$ and each independent vertex
coefficient $-1/3$.
The negative-type inequality becomes
\[
 \frac13+\frac13\,2^q\le1,
\]
which is equivalent to $q\le1$.
\end{proof}

\begin{corollary}\label{cor:q7}
With the notation of \cite{BHKL},
\[
 \boxed{q(7)=1.}
\]
\end{corollary}

\begin{proof}
Combine \cref{thm:P6} with \eqref{eq:correspondence}.
\end{proof}

\section{Further remarks}

The coefficient transport gives a reduction between sign patterns in every
inversion-closed metric class. In particular, on $2k$ points the balanced
$k+k$ pattern follows from the adjacent $(k-1)+(k+1)$ pattern. This explains why the balanced and adjacent complete-split
obstructions have the same exponent in the odd-$n$ part of \eqref{eq:BHKL-conjecture}.

\Cref{prop:one-anchor-subunit} proves every $1+r$ sign pattern for
$0<q\le1$. Consequently, the conjectural seven-point metric statement
\[
 P(7)=2\log_2\frac43
\]
reduces to the $3+4$ pattern: once that pattern is known, the $2+5$
pattern follows from \cref{thm:neighbor-reduction} using $1+6$ and
$3+4$. Similarly, the conjectural eight-point statement
\[
 P(8)=\log_2\frac53
\]
reduces to the $3+5$ pattern. The $2+6$ pattern then follows from $1+7$
and $3+5$, while the $4+4$ pattern follows from
\cref{thm:balanced-reduction}. These cases require an analogue of the four-point correlation estimate
for sign patterns with three positive coefficients.

\appendix
\section{Proof of the geodesic partial-correlation theorem}\label{app:geodesic}

We prove \cref{thm:geodesic-section} by reducing the feasible
geodesic chamber to its triangle and Ptolemy boundary faces. The exact
Bernstein certificate appears in \cref{thm:cube}.

For $1\le q\le2$, define
\[
 \eta_q(t)=\frac{(1+t)^q-1-t^q}{t^{q/2}},\qquad t>0.
\]
The proof of \cref{lem:two-summand} shows
\begin{equation}\label{eq:eta-global}
 0\le\eta_q(t)\le h_q=2^q-2,
 \qquad
 \eta_q(t)=\eta_q(t^{-1}).
\end{equation}
For $1<q<2$, equality in the upper bound holds only at $t=1$.

\subsection{The double-geodesic face and the projective star formula}

\begin{proposition}[Parametrization of the crossing Ptolemy face]\label{prop:mismatch}
Let \(0<\alpha,\theta<1\). Define a metric on \(\{A,B,P,O\}\) by
\[
 \begin{gathered}
 d(A,B)=1,
 \quad d(A,P)=\alpha,
 \quad d(P,B)=1-\alpha,\\
 d(A,O)=\theta,
 \quad d(O,B)=1-\theta,\\
 d(P,O)=\alpha+\theta-2\alpha\theta.
 \end{gathered}
\]
Then \(P\) and \(O\) lie on geodesics from \(A\) to \(B\), and
\[
 d(A,B)d(P,O)
 =d(A,P)d(B,O)+d(A,O)d(P,B).
\]
Conversely, every metric on four distinct points satisfying these three
equalities is a positive scalar multiple of a metric of this form.

Metric inversion at \(A\) gives a three-leaf star centered at \(B\), with branch lengths
\begin{equation}\label{eq:star-branches}
 1,
 \qquad
 p=\frac{1-\alpha}{\alpha},
 \qquad
 o=\frac{1-\theta}{\theta}.
\end{equation}
\end{proposition}

\begin{proof}
The metric and Ptolemy identities follow by direct substitution. For the converse, scale \(d(A,B)\) to one. Under inversion at \(A\), the distances from \(B\) to \(A,P,O\) are the three numbers in \eqref{eq:star-branches}, while the leaf-to-leaf distances are their pairwise sums.
\end{proof}

For $X,Y,Z\in\R$, write
\begin{equation}\label{eq:heron-polynomial}
 \Her(X,Y,Z)=2XY+2YZ+2ZX-X^2-Y^2-Z^2.
\end{equation}
This is the Heron polynomial.

\begin{theorem}[Symmetric projective star formula]\label{thm:star-formula}
Let a star have center $B_0$ and leaves $A_0,P,O$, with positive
branch lengths $a,b,c$, respectively. Invert the metric at $A_0$, and
let $\chi_q(a,b,c)$ be the partial correlation of $A_0$ and $B_0$ based
at $P$ and conditioned on $O$. Put
\[
 A=a^q,
 \quad B=b^q,
 \quad C=c^q,
\]
\[
 U=(a+b)^q,
 \quad V=(a+c)^q,
 \quad W=(b+c)^q.
\]
Define
\[
 \mathcal H_0=\Her(U,V,W),
 \qquad
 \mathcal H_1=\Her(AW,BV,CU),
\]
and
\begin{align*}
 \mathcal N={}&AW(U+V-W)+BV(U+W-V)\\
 &+CU(V+W-U)-2UVW.
\end{align*}
Then
\begin{equation}\label{eq:star-correlation}
 \chi_q(a,b,c)
 =\frac{\mathcal N}{\sqrt{\mathcal H_0\mathcal H_1}}.
\end{equation}
This expression is homogeneous of degree zero and invariant under every permutation of \(a,b,c\).
\end{theorem}

\begin{proof}
In the inverted metric, the powered distances needed for the Schoenberg
matrix are
\[
 d(P,A_0)^q=\frac1U,
 \quad d(P,B_0)^q=\frac{B}{AU},
 \quad d(P,O)^q=\frac{W}{UV},
\]
\[
 d(A_0,B_0)^q=\frac1A,
 \quad d(A_0,O)^q=\frac1V,
 \quad d(B_0,O)^q=\frac{C}{AV}.
\]
Forming the Schoenberg matrix based at $P$ and taking the Schur complement
of the entry indexed by $O$ gives
\[
 S_{A_0A_0}=\frac{\mathcal H_0}{4UVW},
 \qquad
 S_{B_0B_0}=\frac{\mathcal H_1}{4A^2UVW},
 \qquad
 S_{A_0B_0}=\frac{\mathcal N}{4AUVW}.
\]
Their normalized off-diagonal entry is \eqref{eq:star-correlation}.
The resulting expression is homogeneous of degree zero and is unchanged
by a simultaneous permutation of the branch powers and the opposite
pair-sums.
\end{proof}

\subsection{An exact Bernstein cube theorem}

By \cref{thm:star-formula}, order and scale the branches so that
\[
 a=1,
 \qquad
 0<p\leq o\leq1,
\]
where \(p=b/a\) and \(o=c/a\). Set
\begin{equation}\label{eq:xy-chart}
 x=o^{q/2},
 \qquad
 y=(p/o)^{q/2}.
\end{equation}
Then \(x,y\in(0,1]\) and \(p^{q/2}=xy\).

For \(1<q\leq\log_2 3\), set
\begin{equation}\label{eq:defect-coordinates}
 h=h_q,
 \qquad
 e=\frac{\eta_q(o)}h,
 \qquad
 f=\frac{\eta_q(p/o)}h,
 \qquad
 d=\frac{\eta_q(p)}h.
\end{equation}
For \(q=1\), use \(h=0\) and \(e=f=d=0\). \Cref{lem:two-summand} gives
\[
 x,y,h,e,f,d\in[0,1].
\]

For independent variables in this cube, define
\begin{align}
 A_x&=1+x^2+hex,\label{eq:Ax}\\
 A_y&=1+y^2+hfy,\label{eq:Ay}\\
 P&=2x+he+y(xhf-hd),\label{eq:P}\\
 Q&=2y+hf+x(yhe-hd),\label{eq:Q}
\end{align}
and
\begin{align}
 H_x&=4A_xA_y-P^2,\label{eq:Hx}\\
 H_y&=4A_xA_y-Q^2,\label{eq:Hy}\\
 N&=PQ+2hdA_xA_y,\label{eq:N}\\
 \Delta&=\left(\frac{1+2h}{3}\right)^2H_xH_y-N^2.\label{eq:Delta}
\end{align}

\begin{theorem}[Bernstein cube inequality]\label{thm:cube}
For every \((x,y,h,e,f,d)\in[0,1]^6\),
\[
 H_x\geq3,
 \qquad
 H_y\geq3,
 \qquad
 N\geq0,
 \qquad
 \Delta\geq0.
\]
Consequently,
\begin{equation}\label{eq:cube-correlation}
 -\frac{1+2h}{3}
 \leq
 -\frac{N}{\sqrt{H_xH_y}}
 \leq0.
\end{equation}
\end{theorem}

\begin{proof}
Use the tensor-product Bernstein basis on \([0,1]^6\). For a multi-degree \(\mathbf n=(n_1,\ldots,n_6)\), write
\[
 B_{\mathbf i}^{\mathbf n}(\mathbf t)
 =\prod_{j=1}^6
 \binom{n_j}{i_j}t_j^{i_j}(1-t_j)^{n_j-i_j}.
\]
If
\[
 F(\mathbf t)=\sum_{\boldsymbol\alpha}c_{\boldsymbol\alpha}
 \mathbf t^{\boldsymbol\alpha},
\]
then its Bernstein coefficient at \(\mathbf i\) is
\begin{equation}\label{eq:bernstein-conversion}
 b_{\mathbf i}
 =\sum_{\boldsymbol\alpha\leq\mathbf i}
 c_{\boldsymbol\alpha}
 \prod_{j=1}^6
 \frac{\binom{i_j}{\alpha_j}}
 {\binom{n_j}{\alpha_j}}.
\end{equation}
Every basis function is nonnegative, and the basis functions sum to one.

The exact coefficient calculation gives the following data.
\begin{center}
\small
\begin{tabular}{@{}lccccc@{}}
\toprule
Polynomial & Multi-degree & \# coeffs. & Zero & Least positive & Largest\\
\midrule
\(H_x\) & \((2,2,2,2,2,2)\) & 729 & 0 & \(3\) & \(27\)\\
\(H_y\) & \((2,2,2,2,2,2)\) & 729 & 0 & \(3\) & \(27\)\\
\(N\) & \((2,2,3,2,2,2)\) & 972 & 190 & \(1/12\) & \(27\)\\
\(\Delta\) & \((4,4,6,4,4,4)\) & \(21{,}875\) & 415 & \(4/9\) & \(11{,}807/54\)\\
\bottomrule
\end{tabular}
\end{center}
All \(24{,}305\) coefficients are rational and nonnegative. Every
coefficient of \(H_x\) and \(H_y\) is at least $3$. Since the Bernstein
basis functions are nonnegative and sum to one,
\eqref{eq:bernstein-conversion} proves the stated bounds. The ancillary
files contain the complete coefficient list and an independent exact
reconstruction verifier.
\end{proof}

\begin{lemma}[Metric specialization]\label{lem:specialization}
Under \eqref{eq:xy-chart} and \eqref{eq:defect-coordinates}, the quantities in \cref{thm:star-formula} satisfy
\[
 \mathcal H_0=x^2H_x,
 \qquad
 \mathcal H_1=x^4y^2H_y,
 \qquad
 \mathcal N=-x^3yN.
\]
Hence
\begin{equation}\label{eq:chi-cube}
 \chi_q(a,b,c)=-\frac{N}{\sqrt{H_xH_y}}.
\end{equation}
\end{lemma}

\begin{proof}
The three powered pair-sums are
\[
 U=1+x^2y^2+hdxy,
 \qquad
 V=A_x,
 \qquad
 W=x^2A_y.
\]
Substitution into the definitions of \(\mathcal H_0\), \(\mathcal H_1\), and \(\mathcal N\) gives the three identities after expansion.
\end{proof}

\begin{corollary}[Equality faces of the exact certificate]\label{cor:cube-equality}
Under the metric specialization in \cref{lem:specialization}, the zero-face data of the exact Bernstein certificate have the following consequences.
\begin{enumerate}
 \item If \(0<h<1\) and \(\Delta=0\), then
 \[
  x=y=e=f=d=1.
 \]
 \item If \(h=1\) and \(\Delta=0\), then
 \[
  e=f=d=1.
 \]
\end{enumerate}
The ancillary verifier checks these face statements exactly from the stored rational Bernstein coefficients.
\end{corollary}

\subsection{The sharp double-geodesic theorem}

\begin{theorem}[Crossing double-geodesic correlation]\label{thm:double-geodesic}
Let \(d\) be a Ptolemaic metric on the four distinct points
\(\{A,B,P,O\}\). Suppose that \(P\) and \(O\) lie on geodesics from
\(A\) to \(B\) and that the corresponding Ptolemy equality holds. For
\(1\leq q\leq\log_2 3\), let \(\chi_q\) be the partial correlation of
the endpoint vectors after conditioning on the other inserted point. Then
\begin{equation}\label{eq:main-bound}
 -\kappa_q\leq\chi_q<0,
 \qquad
 \kappa_q=\frac{2^{q+1}-3}{3}.
\end{equation}
The lower equality holds exactly when the three branches of the inverted star are equal.
\end{theorem}

\begin{proof}
Invert at one endpoint. \Cref{prop:mismatch} gives a three-leaf
star with positive branch lengths. The projective symmetry in \cref{thm:star-formula} permits the ordered normalization used above. \Cref{thm:cube,lem:specialization} give \eqref{eq:main-bound}.

Suppose first that \(1<q<\log_2 3\). On the metric locus,
$x,y,e,f,d\in(0,1]$ and $h\in(0,1)$. By
\cref{cor:cube-equality}, equality forces
\[
 x=y=e=f=d=1.
\]
Thus $o=1$ and $p/o=1$, so the three branches are equal. At
$q=\log_2 3$, one has $h=1$, and \cref{cor:cube-equality} forces
$e=f=d=1$. Equality in \cref{lem:two-summand} then again gives equal
branch lengths.

At \(q=1\), the symmetric star formula simplifies to
\begin{equation}\label{eq:q-one-star}
 \chi_1(a,b,c)
 =-\sqrt{\frac{abc}{(a+b+c)(ab+bc+ca)}}.
\end{equation}
The inequality
\[
 (a+b+c)(ab+bc+ca)\geq9abc
\]
follows from two applications of the arithmetic-geometric mean inequality,
and equality holds exactly when \(a=b=c\). This completes the equality
analysis. Strict negativity follows from \eqref{eq:q-one-star} at \(q=1\). For
\(q>1\), write
\[
 U=(1+p)^q,\qquad V=(1+o)^q,\qquad W=(p+o)^q.
\]
Direct substitution into \eqref{eq:P}--\eqref{eq:Q} gives
\[
 P=\frac{V+W-U}{x}>0,
 \qquad
 Q=\frac{W+p^qV-o^qU}{(po)^{q/2}}>0.
\]
The first inequality uses $o\ge p$. For the second, set $t=o/p$ and note
that $t(1+p)=t+o\le t+1$, whence
\[
 W+p^qV-o^qU
 =p^q\bigl((1+t)^q+(1+o)^q-t^q(1+p)^q\bigr)>0.
\]
Thus $N=PQ+2hdA_xA_y>0$, which proves strict negativity.
\end{proof}

\begin{corollary}[Crossing boundary at the six-element constant]\label{cor:qstar-boundary}
At \(q=\qstar\), every partial correlation on the crossing
double-geodesic Ptolemy boundary satisfies
\[
 -\frac12\leq\chi_{\qstar}<0.
\]
The lower equality occurs only for the balanced three-leaf star, up to branch permutation and scaling. This equality configuration occurs in a four-point restriction of the balanced complete split metric \(\CS(2,3)\).
\end{corollary}

The cube inequality is stronger than the metric specialization needed here: its six variables vary independently throughout $[0,1]^6$, whereas the defect variables arising from metrics satisfy additional relations.

\subsection{The Ptolemaic geodesic chamber}

Let \(d\) be a metric on the four distinct points \(\{A,P,B,O\}\) such that
\[
 d(A,B)=d(A,P)+d(P,B).
\]
After scaling, write
\[
 d(A,B)=1,
 \qquad d(A,P)=v,
 \qquad d(P,B)=1-v.
\]
The remaining three distances admit an exact coupling parametrization.

\begin{theorem}[Bernoulli-coupling parametrization]\label{thm:coupling-chart}
The metric \(d\) is Ptolemaic if and only if there are unique
parameters $\tau\geq0$, $u\in[0,1]$, and $m\in\R$ satisfying
\[
 |u-v|\leq m\leq u+v-2uv
\]
such that
\begin{equation}\label{eq:coupling-distances}
 \begin{aligned}
 d(A,P)&=v,& d(P,B)&=1-v,\\
 d(O,A)&=\tau+u,& d(O,B)&=\tau+1-u,\\
 d(O,P)&=\tau+m.&
 \end{aligned}
\end{equation}
Put
\[
 \delta=\frac{u+v-2uv-m}{2}.
\]
Then
\begin{equation}\label{eq:coupling-table}
 \Xi=
 \begin{pmatrix}
 uv+\delta&u(1-v)-\delta\\
 (1-u)v-\delta&(1-u)(1-v)+\delta
 \end{pmatrix}
\end{equation}
is a nonnegative \(2\times2\) table with row sums \((u,1-u)\), column sums \((v,1-v)\), and
\[
 \det\Xi=\delta\geq0.
\]
Moreover,
\begin{equation}\label{eq:ptolemy-slack-delta}
 d(A,P)d(O,B)+d(O,A)d(P,B)-d(A,B)d(O,P)=2\delta.
\end{equation}
Thus the upper endpoint \(m=u+v-2uv\) is the Ptolemy-equality face, while the lower endpoint \(m=|u-v|\) is a triangle-equality face.
\end{theorem}

\begin{proof}
Set
\[
 \tau=\frac{d(O,A)+d(O,B)-1}{2},
 \qquad
 u=d(O,A)-\tau,
 \qquad
 m=d(O,P)-\tau.
\]
The triangle inequalities on \(OAB\) give \(\tau\geq0\) and \(u\in[0,1]\). The triangles \(OAP\) and \(OPB\) give
\[
 m\geq |u-v|.
\]
The nontrivial Ptolemy inequality is
\[
 d(A,B)d(O,P)
 \leq d(A,P)d(O,B)+d(O,A)d(P,B),
\]
which becomes
\[
 m\leq u+v-2uv.
\]
The remaining triangle inequalities follow from $\tau\geq0$,
$u,v\in[0,1]$, and $|u-v|\leq m$; the upper bound on $m$ also gives
$m\leq u+v$ and $m\leq2-u-v$. For the three Ptolemy product
inequalities, their slacks are
\[
 u+v-2uv-m,
 \qquad
 2(1-v)\tau+m+u-v,
 \qquad
 2v\tau+m+v-u,
\]
which are nonnegative. Conversely, these calculations verify every
triangle and Ptolemy inequality from the stated parameter bounds.

The entries in \eqref{eq:coupling-table} are nonnegative exactly when
\[
 0\leq\delta\leq\min\{u(1-v),(1-u)v\},
\]
which is equivalent to the interval for \(m\). Direct expansion gives \(\det\Xi=\delta\), and \eqref{eq:ptolemy-slack-delta} follows from \eqref{eq:coupling-distances}.
\end{proof}

The table \(\Xi\) is a coupling of two Bernoulli laws. Its determinant is
the covariance of the associated indicator variables, and \(m\) is their
mismatch probability. As $m$ decreases from its upper endpoint to its lower endpoint, the coupling moves from independence to the upper Fr\'echet--Hoeffding bound.

There is also a compact simplex form. Write the entries of \(\Xi\) as
\(\xi_0,\xi_1,\xi_2,\xi_3\) in row-major order, and divide every distance
and every \(\xi_i\) by \(1+\tau\). Denote the rescaled quantities again by
\(d\) and \(\xi_i\), and put
\[
 \widehat\tau=\frac{\tau}{1+\tau}.
\]
Then
\begin{equation}\label{eq:simplex-chamber}
 \widehat\tau+\xi_0+\xi_1+\xi_2+\xi_3=1,
 \qquad
 \xi_0\xi_3\geq\xi_1\xi_2,
\end{equation}
and
\begin{align*}
 d(A,P)&=\xi_0+\xi_2,&
 d(P,B)&=\xi_1+\xi_3,\\
 d(O,A)&=\widehat\tau+\xi_0+\xi_1,&
 d(O,P)&=\widehat\tau+\xi_1+\xi_2,\\
 d(O,B)&=\widehat\tau+\xi_2+\xi_3.&
\end{align*}
Hence the projective closure of the complete geodesic-insertion domain is a four-simplex cut by one \(2\times2\) determinant inequality.

\begin{remark}[Boundary convention]\label{rem:pseudometric-boundary}
All endpoint arguments below are carried out in the closed Ptolemaic pseudometric chamber. A formal endpoint may have a zero distance because two inserted points coalesce. In that case, the asserted estimate means the limit along positive metrics approaching the endpoint. The line and star bounds used below are uniform, and the relevant Schur quantities extend continuously along these approximating metrics.
\end{remark}

\subsection{An endpoint principle for conditioned correlations}

For a four-point metric on $\{A_0,P,B_0,O\}$, use the powered-distance
notation
\[
 R=d(P,O)^q,
 \quad Y=d(P,A_0)^q,
 \quad Z=d(P,B_0)^q,
\]
\[
 A=d(O,A_0)^q,
 \quad B=d(O,B_0)^q,
 \quad S=d(A_0,B_0)^q.
\]
Fix positive values of \(Y,Z,A,B,S\). For \(R>0\), define
\begin{align}
 U(R)&=4RY-(R+Y-A)^2,\label{eq:UR}\\
 V(R)&=4RZ-(R+Z-B)^2,\label{eq:VR}\\
 W(R)&=2R(Y+Z-S)-(R+Y-A)(R+Z-B),\label{eq:WR}
\end{align}
and, whenever \(U(R),V(R)>0\),
\begin{equation}\label{eq:chiR}
 \chi(R)=\frac{W(R)}{\sqrt{U(R)V(R)}}.
\end{equation}

\begin{theorem}[Endpoint principle]\label{thm:endpoint-principle}
On every interval on which \(U,V>0\), an interior critical point of \(\chi\) with \(\chi<0\) is a strict local maximum. Consequently, the negative minimum of \(\chi\) on a compact feasible interval is attained at an endpoint.
\end{theorem}

\begin{proof}
Put \(R=e^t\),
\[
 d=\binom{Y-A}{Z-B},
 \qquad
 H=\begin{pmatrix}Y&(Y+Z-S)/2\\(Y+Z-S)/2&Z\end{pmatrix}.
\]
The Schur complement obtained by conditioning on the coordinate of squared length \(R\) is
\begin{equation}\label{eq:schur-pencil}
 K(t)=K_0-u(t)u(t)^{\mathsf T}-v(t)v(t)^{\mathsf T},
\end{equation}
where
\[
 u(t)=\frac{e^{t/2}}2\binom11,
 \qquad
 v(t)=\frac{e^{-t/2}}2d,
\]
and \(K_0\) is independent of \(t\). Therefore
\[
 K'=-uu^{\mathsf T}+vv^{\mathsf T},
 \qquad
 K''=-uu^{\mathsf T}-vv^{\mathsf T}.
\]
At a fixed interior point, apply a fixed positive diagonal congruence
so that
\[
 K=\begin{pmatrix}1&-r\\-r&1\end{pmatrix},
 \qquad r>0.
\]
After this fixed congruence, write
$K(t)=\bigl(\begin{smallmatrix}a(t)&c(t)\\c(t)&b(t)\end{smallmatrix}\bigr)$;
at the chosen point, $a=b=1$ and $c=-r$. Write the transformed vectors
as \(u=(u_1,u_2)^{\mathsf T}\) and \(v=(v_1,v_2)^{\mathsf T}\). The
components \(u_1,u_2\) are positive, and
\[
 \ell=\log(-c)-\frac12\log a-\frac12\log b=\log(-\chi).
\]
At a critical point of \(\chi\), one has \(\ell'=0\). Put
\[
 U_0=u_1^2+u_2^2,
 \qquad
 V_0=v_1^2+v_2^2,
 \qquad
 \pi=u_1u_2,
 \qquad
 \nu=v_1v_2.
\]
The critical equation is
\[
 \pi-\nu=\frac r2(V_0-U_0).
\]
A direct differentiation gives
\begin{equation}\label{eq:ell-second}
 \ell''
 =U_0+\frac{2\pi}{r}
 +\frac14\left[(v_1^2-u_1^2)-(v_2^2-u_2^2)\right]^2>0.
\end{equation}
Since \(\chi=-e^{\ell}\), it follows that \(\chi''=\chi\ell''<0\). Thus every negative interior critical point is a strict local maximum.
\end{proof}

The endpoint principle concerns this $2\times2$ Schur-complement pencil and does not use Ptolemy geometry. The geometry enters through the endpoint classification in \cref{thm:coupling-chart}.

\subsection{Star and line endpoint estimates}

\begin{proposition}[Three-leaf star]\label{prop:star-elementary}
Let \(d\) be a three-leaf star with center \(P\) and positive branch lengths. For \(1\leq q\leq\log_2 3\), the partial correlation of two leaf vectors after conditioning on the third satisfies
\begin{equation}\label{eq:star-elementary-bound}
 -\frac{h_q}{2-h_q}\leq\chi_q\leq0.
\end{equation}
If $q>1$, then $\chi_q<0$, and equality in the lower bound holds
exactly for the balanced star. If $q=1$, then $\chi_1=0$ for every
three-leaf star.
\end{proposition}

\begin{proof}
For branch lengths $r,s>0$, define the normalized convexity defect
\[
 \delta_q(r,s)
 =\frac{(r+s)^q-r^q-s^q}{2(rs)^{q/2}}.
\]
Let $\alpha,\beta,\gamma$ be the defects for the three branch pairs.
\Cref{lem:two-summand} gives
\[
 0\leq\alpha,\beta,\gamma\leq\frac{h_q}{2}.
\]
The three pairwise correlations of the leaf vectors based at the center are \(-\alpha,-\beta,-\gamma\). Conditioning on the third leaf gives
\[
 \chi_q=\frac{-\gamma-\alpha\beta}
 {\sqrt{(1-\alpha^2)(1-\beta^2)}}.
\]
Hence
\[
 \chi_q\geq
 -\frac{h_q/2+h_q^2/4}{1-h_q^2/4}
 =-\frac{h_q}{2-h_q}.
\]
For $q>1$, all three defects are positive, so $\chi_q<0$. Equality in the lower bound forces equality in all three defect bounds, and \cref{lem:two-summand} then forces the branch lengths to be equal. At $q=1$, all three defects vanish, so $\chi_1=0$.
\end{proof}

\begin{proposition}[Four points on a line]\label{prop:line-bound}
Let \(A,O,P,B\) occur in this order on a line, with positive successive
gaps \(\ell,m,n\). For \(1\leq q\leq\log_2 3\), the partial correlation
satisfies
\begin{equation}\label{eq:line-bound}
 -\frac{2h_q}{4-h_q^2}
 \leq\chi_q\leq
 \frac{h_q^2}{4-h_q^2}.
\end{equation}
The same estimate holds for every ordering in which \(P\) and \(O\) lie on the geodesic from \(A\) to \(B\).
\end{proposition}

\begin{proof}
Put
\begin{align*}
 e&=(\ell+m)^q-\ell^q-m^q,\\
 f&=(m+n)^q-m^q-n^q,\\
 k&=(\ell+m+n)^q-(\ell+m)^q-(m+n)^q+m^q.
\end{align*}
The mixed difference has the integral representation
\[
 k=q(q-1)\int_0^{\ell}\int_0^n(m+s+t)^{q-2}\,dt\,ds.
\]
Thus \(k\geq0\), and it decreases as the separating gap \(m\) increases. Hence $k\le(\ell+n)^q-\ell^q-n^q$, and \cref{lem:two-summand} gives
\[
 e\leq h_q(\ell m)^{q/2},
 \qquad
 f\leq h_q(mn)^{q/2},
 \qquad
 k\leq h_q(\ell n)^{q/2}.
\]
A direct Schur-complement calculation gives
\[
 U=4\ell^qm^q-e^2,
 \qquad
 V=4m^qn^q-f^2,
 \qquad
 W=ef-2m^qk.
\]
Consequently,
\[
 \sqrt{UV}\geq(4-h_q^2)m^q(\ell n)^{q/2},
\]
while
\[
 -2h_qm^q(\ell n)^{q/2}
 \leq W\leq
 h_q^2m^q(\ell n)^{q/2}.
\]
Division proves \eqref{eq:line-bound}. Exchanging \(A\) and \(B\) covers the other ordering.
\end{proof}

The elementary comparisons
\begin{equation}\label{eq:endpoint-constant-comparison}
 \frac{h_q}{2-h_q}\leq\kappa_q,
 \qquad
 \frac{2h_q}{4-h_q^2}\leq\kappa_q,
 \qquad
 \frac{h_q^2}{4-h_q^2}\leq\frac13\leq\kappa_q
\end{equation}
hold for \(0\leq h_q\leq1\).

\begin{lemma}[Shared geodesic segment]\label{lem:shared-segment}
Suppose
\[
 d(A,B)=d(A,P)+d(P,B),
 \qquad
 d(O,B)=d(O,P)+d(P,B).
\]
Then, for \(1\leq q\leq\log_2 3\),
\[
 -\kappa_q\leq\chi_q\leq\frac{h_q^2}{4-h_q^2}.
\]
\end{lemma}

\begin{proof}
Fix \(d(A,P),d(P,B),d(O,P)\) and vary \(a=d(A,O)\) over its triangle interval. Write
\[
 x=d(A,P),\qquad b=d(P,B),\qquad c=d(O,P).
\]
Since \(d(A,B)=x+b\) and \(d(O,B)=c+b\), the three Ptolemy slacks are
\[
 b(a+x-c),\qquad b(x+c-a)+2xc,\qquad b(a+c-x),
\]
so the entire closed triangle interval \( |x-c|\le a\le x+c\) remains Ptolemaic. If an endpoint is a pseudometric, we use the convention in \cref{rem:pseudometric-boundary}. Put
\[
 \rho=\frac{R+Y-A}{2\sqrt{RY}},
 \qquad
 \beta=\frac{R+Z-B}{2\sqrt{RZ}},
 \qquad
 \gamma=\frac{Y+Z-S}{2\sqrt{YZ}}.
\]
Here \(\beta,\gamma\leq0\), and
\[
 \chi_q(\rho)=
 \frac{\gamma-\rho\beta}
 {\sqrt{(1-\rho^2)(1-\beta^2)}}.
\]
Its derivative has the sign of \(\gamma\rho-\beta\). Any interior critical point is therefore a local maximum. The two endpoints are a three-leaf star and a line metric, so \Cref{prop:star-elementary,prop:line-bound}, together with \eqref{eq:endpoint-constant-comparison}, give the lower bound.

If \(\chi_q>0\), write \(\beta=-b_0\), \(\gamma=-c_0\) with \(b_0,c_0\geq0\). Positivity gives \(b_0\rho>c_0\), and hence \(b_0-c_0\rho>0\). Thus \(\chi_q\) increases with \(\rho\) throughout its positive region, so its positive maximum occurs at the line endpoint. \Cref{prop:line-bound} gives the upper bound.
\end{proof}

\begin{lemma}[Positive correlations reduce to a line]\label{lem:positive-to-line}
For every Ptolemaic metric on four distinct points with a geodesic
insertion, and every
\(1\leq q\leq\log_2 3\),
\[
 \chi_q>0
 \quad\Longrightarrow\quad
 \chi_q\leq\frac{h_q^2}{4-h_q^2}.
\]
\end{lemma}

\begin{proof}
Put
\[
 \rho_A=\frac{R+Y-A}{2\sqrt{RY}},
 \qquad
 \rho_B=\frac{R+Z-B}{2\sqrt{RZ}},
 \qquad
 \rho_0=\frac{Y+Z-S}{2\sqrt{YZ}}\leq0.
\]
Then
\[
 \chi_q=
 \frac{\rho_0-\rho_A\rho_B}
 {\sqrt{(1-\rho_A^2)(1-\rho_B^2)}}.
\]
The two numbers \(\rho_A,\rho_B\) have opposite signs. Suppose
\(\rho_A>0>\rho_B\); the other case is symmetric. For fixed
\(\rho_A,\rho_0\),
\[
 \frac{\partial\chi_q}{\partial\rho_B}
 =\frac{\rho_0\rho_B-\rho_A}
 {\sqrt{1-\rho_A^2}(1-\rho_B^2)^{3/2}}<0.
\]
Indeed, \(\chi_q>0\) gives
\(\rho_A(-\rho_B)>-\rho_0\), which implies
\(\rho_0\rho_B<\rho_A\). Increasing \(d(O,B)\) decreases
\(\rho_B\), so \(\chi_q\) is bounded above by its value at
\[
 d(O,B)=d(O,P)+d(P,B).
\]
To see that the replacement remains Ptolemaic, write
\[
 x=d(A,P),\qquad b=d(P,B),\qquad c=d(O,P),\qquad a=d(A,O).
\]
After setting \(d(O,B)=c+b\) and using \(d(A,B)=x+b\), the three Ptolemy slacks are
\[
 b(a+x-c),\qquad b(x+c-a)+2xc,\qquad b(a+c-x),
\]
which are nonnegative by the triangle inequalities on \(OAP\). On the resulting shared-segment face, the proof of \cref{lem:shared-segment} moves every positive maximum to the line endpoint. \Cref{prop:line-bound} proves the result.
\end{proof}

\subsection{Completion of the geodesic-insertion theorem}

Metric inversion at the inserted point preserves the partial correlation.

\begin{lemma}[Inversion covariance]\label{lem:inversion}
Invert a Ptolemaic metric on distinct points at \(P\):
\[
 \widehat d(P,X)=\frac1{d(P,X)},
 \qquad
 \widehat d(X,Y)=
 \frac{d(X,Y)}{d(X,P)d(Y,P)}.
\]
Then \(\widehat d\) is Ptolemaic and
\[
 \widehat\chi_q=\chi_q.
\]
If the original metric satisfies
\[
 d(A,B)d(O,P)
 =d(A,P)d(O,B)+d(O,A)d(P,B),
\]
then both \(P\) and \(O\) lie on geodesics from \(A\) to \(B\) in the inverted metric.
\end{lemma}

\begin{proof}
Metric inversion preserves Ptolemy inequalities. In the powered variables, it sends
\[
 (Y,Z,A,B,R,S)
 \longmapsto
 \left(
 Y^{-1},Z^{-1},\frac{A}{YR},\frac{B}{ZR},R^{-1},\frac{S}{YZ}
 \right).
\]
Substitution into \eqref{eq:UR}--\eqref{eq:WR} gives
\[
 \widehat U=\frac{U}{R^2Y^2},
 \qquad
 \widehat V=\frac{V}{R^2Z^2},
 \qquad
 \widehat W=\frac{W}{R^2YZ},
\]
which proves correlation invariance.

The original geodesic relation gives
\[
 \widehat d(A,B)=\widehat d(A,P)+\widehat d(P,B).
\]
The displayed Ptolemy equality gives
\[
 \widehat d(A,O)+\widehat d(O,B)
 =\frac{d(O,A)d(P,B)+d(A,P)d(O,B)}
 {d(A,P)d(O,P)d(P,B)}
 =\widehat d(A,B).
\]
\end{proof}

\begin{theorem}[Complete double-geodesic face]\label{thm:full-double-geodesic}
Suppose that \(P\) and \(O\) both lie on geodesics from \(A\) to \(B\). Then, for \(1\leq q\leq\log_2 3\),
\begin{equation}\label{eq:full-double-bound}
 -\kappa_q\leq\chi_q\leq\frac{h_q^2}{4-h_q^2}.
\end{equation}
\end{theorem}

\begin{proof}
Normalize \(d(A,B)=1\), and let \(u,v\in[0,1]\) be the positions of \(O,P\). \Cref{thm:coupling-chart} gives the complete interval
\[
 |u-v|\leq d(O,P)\leq u+v-2uv.
\]
By \cref{thm:endpoint-principle}, the negative minimum occurs at an endpoint. The lower endpoint is a line metric. The upper endpoint is the crossing Ptolemy-equality metric treated in \cref{thm:double-geodesic}. If an endpoint lies only in the pseudometric closure, approximate it by positive metrics and pass to the limit as in \cref{rem:pseudometric-boundary}. These give the lower bound. \Cref{lem:positive-to-line} gives the upper bound.
\end{proof}

\begin{theorem}[Sharp geodesic-insertion correlation]\label{thm:geodesic-insertion}
Let \(d\) be a Ptolemaic metric on the four distinct points
\(\{A,P,B,O\}\) with
\[
 d(A,B)=d(A,P)+d(P,B).
\]
For every \(1\leq q\leq\log_2 3\),
\begin{equation}\label{eq:geodesic-main}
 -\kappa_q
 \leq\chi_q
 \leq\lambda_q,
 \qquad
 \lambda_q=\frac{h_q^2}{4-h_q^2}.
\end{equation}
The lower constant is sharp. At \(q=\qstar\),
\begin{equation}\label{eq:qstar-main}
 -\frac12\leq\chi_{\qstar}\leq\frac1{63}.
\end{equation}
Equality on the left in \eqref{eq:qstar-main} occurs only when the Ptolemy-boundary inversion in \cref{lem:inversion} produces the balanced crossing double-geodesic configuration, up to scaling and permutation of the three star branches.
\end{theorem}

\begin{proof}
Fix the four exterior distances \(d(A,P),d(P,B),d(O,A),d(O,B)\). \Cref{thm:coupling-chart} identifies the feasible interval for \(d(O,P)\). By \cref{thm:endpoint-principle}, a negative minimum occurs at an endpoint. If the endpoint has a zero distance, we interpret the argument by positive approximation in the closed pseudometric chamber, as in \cref{rem:pseudometric-boundary}.

At the triangle endpoint, one of
\[
 d(O,B)=d(O,P)+d(P,B),
 \qquad
 d(O,A)=d(O,P)+d(P,A)
\]
holds, so \Cref{lem:shared-segment} applies. At the Ptolemy-equality endpoint, invert at \(P\). \Cref{lem:inversion} preserves \(\chi_q\) and produces a complete double-geodesic metric, so \Cref{thm:full-double-geodesic} applies. This proves the lower bound.

\Cref{lem:positive-to-line} proves the upper bound.

The crossing family in \cref{thm:double-geodesic} attains \(-\kappa_q\). At \(q=\qstar\), the star and line endpoint bounds are strictly larger than \(-1/2\), while every negative interior critical point is a strict local maximum. Hence equality can occur only when the Ptolemy-boundary inversion produces the balanced crossing configuration characterized in \cref{thm:double-geodesic}, up to scaling and permutation of the three star branches.
\end{proof}

\end{document}